\documentclass[a4paper,12pt]{amsart}
\usepackage[utf8]{inputenc}
\usepackage{amsmath,amssymb,amsfonts,amsthm,graphicx}

\newtheorem{theorem}{Theorem}[section]
\newtheorem{lemma}[theorem]{Lemma}

\numberwithin{equation}{section}

\theoremstyle{remark}

\newcommand{\Ric}{\mathop{\mathrm{Ric}}}
\newcommand{\dist}{\mathop{\mathrm{dist}}\nolimits}

\title{Gradient estimates for the heat equation under the Ricci-Harmonic Map flow}

\author{Mihai B\u{a}ile\c{s}teanu}
\thanks{Department of Mathematics, University of Rochester, 801 Hylan Bld, Rochester, NY 14627, USA \texttt{mbailest@z.rochester.edu} \\}

\begin{document}

\begin{abstract}
The paper establishes a series of gradient estimates for positive solutions to the heat equation  on a manifold $M$ evolving under the Ricci flow, coupled with the harmonic map flow between $M$ and a second manifold $N$. We prove Li-Yau type Harnack inequalities and we consider the cases when $M$ is a complete manifold without boundary and when $M$ is compact, without boundary.  
\end{abstract}

\maketitle

\section{Introduction}

The present article focuses on the behaviour of positive solutions to the heat equation on a Riemannian manifold $M$, whose metric evolves under the Ricci flow coupled with the harmonic map flow of a map $\phi$ from $M$ to another manifold $N$. Our goal is to establish a series of Li-Yau type gradient estimates, which naturally lead to Harnack inequalities. The results are both of local and global nature, as we are studying the case when $M$ is complete without boundary and the case when $M$ is compact without boundary.   

Given two Rimeannian manifolds $(M,g)$ and $(N,\gamma)$ and a map $\phi:M\to N$, the Ricci-harmonic map flow is the coupled system of the Ricci flow with the harmonic map flow of $\phi$: \begin{equation}\label{RH_flow_intro}
\begin{cases}\frac\partial{\partial t} g(x,t)=-2\Ric(x,t)+2\alpha(t)\nabla\phi(x,t)\otimes\nabla\phi(x,t)\\
\frac{\partial}{\partial t}\phi(x,t)=\tau_g\phi(x,t)
\end{cases}
 \end{equation} 
where $\alpha$ is a positive coupling time-dependent function. $\tau_g\phi$ denotes the tension field of the map $\phi$ with respect to the metric $g(t)$. We will refer to it as the $(RH)_\alpha$ flow, for short and we call $(g(x,t), \phi(x,t))$ with $t\in[0,T]$ a solution to this flow.  It is interesting that it may be less singular than both the Ricci flow (to which it reduces when $\alpha(t)=0$) and the harmonic map flow. We assume that the curvature of $M$ remains bounded for all $t\in [0,T]$. We also consider a function $u:M\times [0,T]\to (0,\infty)$ which solves the heat equation
\begin{equation}\label{heat_intro}
\left(\triangle-\frac{\partial}{\partial t}\right)u(x,t)=0,\, x\in M,\, t\in[0,T].
\end{equation} 

One can give a simple physical interpretation of problem \eqref{RH_flow_intro} combined with \eqref{heat_intro}: the manifold $M$ with initial metric $g(x,0)$ can be thought as an object having the temperature distribution $u(x,0)$. Then, simultaneously, $M$ starts evolving under the $(RH)_\alpha$ flow, while heat is spreading onto it. The solution $u(x,t)$ represents the temperature at time $t$ and point $x$. 

The $(RH)_\alpha$ flow was introduced in \cite{RM12}, where the author proved short time existence and studied energy and entropy functionals, existence of singularities, a local non-collapsing property etc. Moreover, a version of this flow appeared earlier in the work of List \cite{BL08}, where the case of $\phi$ being a scalar function and $\alpha=2$ was analyzed, and where it was shown to be equivalent to the gradient flow of an entropy functional, whose stationary points are solutions to the static Einstein vacuum equations. 

The $(RH)_\alpha$ flow also arises when one studies Ricci flow on warped product spaces. Given a warped product metric $g_M=g_N+e^{2\phi}g_F$ on a manifold $M=N\times F$ (where $\phi\in C^{\infty}(N)$), the Ricci flow equation on $M$ $\frac{\partial g_M}{\partial t}=-2\Ric_M$ leads to the following equations on each component: \[\begin{cases}
\frac{\partial g_N}{\partial t}=-2\Ric_N+2m\ d\phi\otimes\ d\phi \\
\frac{\partial \phi}{\partial t}=\triangle \phi-\mu e^{-2\phi}                                                                                                                                                                                                                                                                                                                                                                                                                                                                                                                                                                                                                                                       \end{cases}\] 
if the fibers $F$ are $m$-dimensional and $\mu$-Einstein. This is a particular version of the $(RH)_\alpha$, where the target manifold is one dimensional, and has been studied by M. B. Williams in \cite{MBW13} and by H. Tran in \cite{HT12} (when $\mu=0$).  

The scalar curvature of a manifold evolving under the $(RH)_\alpha$ flow satisfies the heat equation with a potential (depending on the Ricci curvature of $M$, the map $\phi$ and the Riemann curvature tensor of $N$), so in order to understand the behavior of the metric under the $(RH)_\alpha$ flow one needs to study the heat equation. 

The study of the heat equation under the Ricci flow started with R. Hamilton's paper \cite{RH95}, and later it was pursued in \cite{QZ06,LN04,CG02,XCRH09,RH95}. Most recently, gradient estimates for the heat equation under the Ricci flow were analyzed in \cite{MBXCAP09}, \cite{SL09} and \cite{JS11}.

The main two results of this paper are the following two theorems.

\begin{theorem}\label{thm_sp-tm-global}
Let $\big(g(x,t),\phi(x,t)\big)_{t\in[0,T]}$ is a solution to the $(RH)_\alpha$
flow \eqref{RH_flow_intro}, where the manifold $M^n$ is compact and $\alpha(t)$ is a non-increasing function in time, bounded from below by $\bar\alpha$. Assume that $0\le\Ric(x,t)\leq kg(x,t)$ for some $k>0$ and all $(x,t)\in M\times[0,T]$. Moreover suppose that $0\leq\nabla_i\phi\nabla_j\phi\leq \frac{C}{t}g_{ij}$, for all $i,j\in\{1,2,...,n\}$ ($C$ being a constant depending on $n$ and $\bar\alpha$).  Let $u$ be a smooth
positive function $u:M\times[0,T]\to\mathbb R$ satisfying the heat
equation $u_t=\triangle u$. The estimate
\begin{align}\label{sp-tm-glob} \frac{|\nabla
u|^2}{u^2}-\frac{u_t}{u}\leq kn + {C_n\over 2 t}
\end{align} holds for all $(x,t)\in M\times(0,T]$, where $C_n=n/2+4nC\alpha(0)$.
\end{theorem}

This estimate resembles the Li-Yau inequality from \cite{PLSTY86}, where the authors found that the solution to the heat equation \eqref{heat_intro} on a static manifold (with non-negative Ricci curvature) satisfies: 
\[\frac{|\nabla u|^2}{u^2}-\frac{u_t}{u}\leq \frac{n}{2t}\]
and the expression becomes equality when $u(x,t)$ is the heat kernel on the Euclidean space. 

The second result is local in nature and is achieved inside the ball $B_{\rho, T}=\{(x,t)\in M\times[0,T]|\dist(x,x_0,t)<\rho\}$, while requiring $M$ only to be complete:

\begin{theorem}\label{thm_sp-tm-local}
Let $\big(g(x,t),\phi(x,t)\big)_{t\in[0,T]}$ be a solution to the
$(RH)_\alpha$ flow~\eqref{RH_flow_intro}, where $M^n$ is complete and $\alpha(t)$ is a non-increasing function in time, bounded from below by $\bar\alpha$. Suppose $-k_1g(x,t)\le\Ric(x,t)\leq
k_2g(x,t)$ for some $k_1,k_2>0$ and all $(x,t)\in B_{\rho,T}$ and that $0\leq\nabla_i\phi\nabla_j\phi\leq \frac{C}{t}g_{ij}$, for all $i,j\in\{1,2,...,n\}$ ($C$ being a constant depending on $n$ and $\bar\alpha$). Consider a smooth positive function $u:M\times[0,T]\to\mathbb R$ solving the heat equation $\triangle u=u_t$. There exists a constant
$C'$ that depends on the dimension of $M$, $\alpha(0)$, and $C$ and satisfies the
estimate
\begin{align}\label{sp-tm-loc}
\frac{|\nabla u|^2}{u^2}-\beta\frac{u_t}u\leq C'\beta^2\cdot\left(\frac{\beta^2}{\rho(\beta-1)}+\frac{1}{t}+\max\{k_1,k_2\}\right)+\frac{n\beta k_1}{4(\beta-1)}
\end{align}
for all $\beta>1$ and all $(x,t)\in B_{\frac\rho2,T}$ with $t\ne0$.
\end{theorem}

The proof of the theorems are given in section \ref{global} and \ref{local}, respectively. We conclude the paper with section \ref{trei}, where we present the Harnack inequalities resulting from integrating the above gradient estimates over space-time paths.  

\section{Gradient estimates}\label{doi}

\subsection{Preliminaries}\label{prelim}

First, let us note that $M$ is a connected, oriented, smooth, $n$-dimensional Riemannian manifold, without boundary, with metric $g$. Similarly $N$ is also a smooth manifold, without boundary, of dimension $m$. We assume that $N$ is isometrically embedded into the Euclidean space $\mathbb{R}^d$ (which follows by Nash's embedding theorem) for large enough $d$, so one may write $\phi=(\phi^\mu)_{1\leq\mu\leq d}$. 

We will assume that $(\phi(x,t),g(x,t))$ is a solution to the $(RH)_\alpha$ flow \eqref{RH_flow_intro}, for $x\in M$ and $t\in[0,T]$, where $T<T_\epsilon$, where $T_\epsilon$ is the time when there is possibly a blowup in the metric. 

The tensor $\nabla\phi(x,t)\otimes\nabla\phi(x,t)$ has the following expression in local coordinates: $(\nabla\phi\otimes\nabla\phi)_{ij}=\nabla_i\phi^\mu\nabla_j\phi^\mu$ and the energy density of the map $\phi$ is given by $|\nabla\phi|^2=g^{ij}\nabla_i\phi^\mu\nabla_j\phi^\mu$, where we use the convention that repeated Latin indices are summed over from $1$ to $n$, while the Greek are summed from 1 to $d$. All the norms are taken with respect to the metric $g$ at time $t$. 

As we mentioned in the introduction, we will assume the most general condition for the coupling function $\alpha(t)$, as it appears in \cite{RM12}: it is a non-increasing function in time, bounded from below by $\bar\alpha>0$, at any time. 

Note that in \cite{RM12} (Proposition 5.5), the author proves that if $N$ has non-positive sectional curvature, then 
\[|\nabla\phi(x,t)|^2\leq \frac{n}{2\bar\alpha t},\, \forall (x,t)\in M\times(0,T]\]
where $n$ is the dimension of $M$ and $[0,T]$ is the interval in which the flow is defined ($T<T_e$, where $T_e$ is the moment when there is possibly a blowup). 

Both of our theorems hold under a slightly stronger assumption: $0\leq\nabla_i\phi\nabla_j\phi\leq \frac{C}{t}g_{ij}$, for all $i,j\in\{1,2,...,n\}$ ($C$ being a constant depending on $n$ and $\bar\alpha$), which immediately implies the above inequality. As a future endeavor, one may want to prove the same theorems without this assumption. 

Following the notation in \cite {RM12}, it will be easier to introduce these quantities: 
\begin{align*}
\mathcal{S}&:= \Ric-\alpha\nabla\phi\otimes\nabla\phi\\
S_{ij}&:=R_{ij}-\alpha\nabla_i\phi\nabla_j\phi\\
S&:=R-\alpha|\nabla\phi|^2
\end{align*}

Both proofs are based on the following lemma: 

\begin{lemma}\label{lemma_LY}
Suppose $\big(g(x,t),\phi(x,t)\big)_{t\in[0,T]}$ is a solution to
the $(RH)_\alpha$ flow \eqref{RH_flow_intro}. Assume that
$-k_1g(x,t)\le\Ric(x,t)\leq k_2g(x,t)$ for some $k_1,k_2>0$ and that $0\leq\nabla_i\phi\nabla_j\phi\leq \frac{C}{t}g_{ij}$, for all $i,j\in\{1,2,...,n\}$ ($C$ being a constant depending on $n$ and $\bar\alpha$). Suppose $u:M\times[0,T]\to\mathbb R$ is a
smooth positive function satisfying the heat
equation $u_t=\triangle u$. Given $\beta\ge1$, define $f=\log u$ and
$F=t\left(|\nabla f|^2-\beta f_t\right)$. The estimate
\begin{align}\label{est_lem_LY}
\left(\triangle-\frac{\partial}{\partial t}\right) F\geq &-2\nabla
f\nabla F+\frac{2a\beta t}{n}\left(|\nabla f|^2-f_t\right)^2 \notag \\ &-\left(|\nabla f|^2-\beta f_t\right)-2k_1\beta t|\nabla
f|^2-\frac{\beta tn}{2b}\max\left\{k_1^2,k_2^2\right\} \notag \\
 &-\frac{\beta\alpha^2(0)n}{2b}\frac{C^2}{t} - 2(\beta-1)\alpha(0) C|\nabla f|^2
\end{align} holds for
any $a,b>0$ such that $a+2b=\frac{1}{\beta}$.
\end{lemma}

\begin{proof}
Since $u_t=\triangle u$ and $f=\log u$, then $f_t=\triangle f+|\nabla f|^2$. Also, note that 
\begin{align*}
(|\nabla f|^2)_t&=2S_{ij}\nabla_if\nabla_jf+2\nabla f\cdot \nabla(f_t)\\ 
(\triangle f)_t&=2S_{ij}\nabla_i\nabla_j f+\triangle (f_t) \\
\triangle\nabla_i f&=\nabla_i\triangle f+R_{ij}\nabla_j f\\
\triangle|\nabla f|^2&=2|\nabla\nabla f|^2+2R_{ij}\nabla_if\nabla_j f+2\nabla_i f\nabla_i(\triangle f)
\end{align*}
The first two follow from the flow equation, while the last two are Bochner identities.

Let's denote $\nabla_i f:=f_i$ and $\nabla_i\nabla_jf:=f_{ij}$.

Using these, one gets that 
\[\triangle F-F_t=-2\nabla f\cdot \nabla F-(|\nabla f|^2-\beta f_t)+2t(f_{ij}^2+\beta S_{ij}f_{ij})+2tR_{ij}f_if_j+2t(\beta-1)S_{ij}f_if_j\]

The last two terms can be bounded as follows:
\begin{align*}
2tR_{ij}f_if_j+2t(\beta-1)S_{ij}f_if_j & = 2t\beta R_{ij}f_if_j-2t(\beta-1)\alpha(t)\nabla_i\phi\nabla_j\phi f_if_j  \\
 & \geq -2t\beta k_1 |\nabla f|^2-2(\beta-1)\alpha(0) C |\nabla f|^2
\end{align*}

Suppose at time $t$, at the point $x$ where we are estimating the quantity, we choose normal coordinates associated to the Levi-Civita connection. Then at this point the metric $g$ is diagonal (and hence so is $\Ric$), in fact it is the identity, and has zero directional derivatives. However, that doesn't mean that $\nabla\phi\otimes\nabla\phi$ can also be diagonalized. But  this point $g_{ij}=g^{ij}$ and hence $\alpha(t)\nabla_i\phi\nabla_j\phi f_if_j\leq \alpha(0)\frac{C}{t}g^{ij}f_if_j=\alpha(0) C |\nabla f|^2$.

The term $f_{ij}^2+\beta S_{ij}f_{ij}$ needs to be dealt separately:
\begin{align*}
f_{ij}^2+\beta S_{ij}f_{ij}& =f_{ij}^2+\beta R_{ij}f_{ij}-\beta\alpha\nabla_i\phi\nabla_j\phi f_{ij}\\
  & =(a+2b)\beta f_{ij}^2+\beta R_{ij}f_{ij}-\beta\alpha\nabla_i\phi\nabla_j\phi f_{ij} \\
  & = a\beta f^2_{ij}+\beta\left(\sqrt{b}f_{ij}+\frac{R_{ij}}{2\sqrt{b}}\right)-\frac{\beta}{4b}R_{ij}^2+\beta\left(\sqrt{b}f_{ij}+\frac{\alpha\nabla_i\phi\nabla_j\phi}{2\sqrt{b}}\right)-\frac{\beta}{4b}(\alpha\nabla_i\phi\nabla_j\phi)^2 \\
  &\geq a\beta f^2_{ij} -\frac{\beta}{4b}R_{ij}^2-\frac{\beta}{4b}(\alpha\nabla_i\phi\nabla_j\phi)^2 \\
  & \geq  a\beta f^2_{ij}-\frac{\beta nk^2}{4b}-\frac{\beta\alpha^2(0) n C^2}{4bt^2}
\end{align*}
where $k=\max\{k_1,k_2\}$.

Note that because $\nabla_i\phi\nabla_j\phi\leq \frac{C}{t}g_{ij}$, then $(\nabla_i\phi\nabla_j\phi)^2\leq \frac{C}{t}g_{ij}g^{ij}=\frac{nC^2}{t^2}$ since we are using normal coordinates ($g_{ij}=g^{ij}$). 

Finally, let's observe that $\sum f_{ij}^2\geq\frac{(\triangle f)^2}{n}=\frac{(f_t-|\nabla f|^2)^2}{n}$. 

Putting all together, we obtain: 

\begin{align*}
\triangle F-F_t & \geq -2\nabla f\cdot \nabla F-(|\nabla f|^2-\beta f_t) +\frac{2a\beta t}{n}\left(|\nabla f|^2-f_t\right)^2 -2k_1\beta t|\nabla
f|^2\\
 &-\frac{\beta tn}{2b}\max\left\{k_1^2,k_2^2\right\}-\frac{\beta\alpha^2(0)n}{2b}\frac{C^2}{t} - 2(\beta-1)\alpha(0) C|\nabla f|^2  
\end{align*}
\end{proof}

The proofs of Theorem \ref{thm_sp-tm-local} will involve a cut-off function on
$B_{\rho,T}$. The construction of this function will rely on the following well-known lemma. It was previously used in the proofs of Theorems~2.3 and~3.1 in~\cite{QZ06}; see also~\cite[Chapter~IV]{RSSTY94} and~\cite{PSQZ06}.

\begin{lemma}\label{lem_cutoff}
Given $\tau\in(0,T]$, there exists a smooth function
$\bar\Psi:[0,\infty)\times[0, T]\to\mathbb R$ satisfying the
following requirements:
\begin{enumerate}
\item
The support of $\bar\Psi(r,t)$ is a subset of $[0,\rho]\times[0,T]$,
and $0\leq\bar\Psi(r,t)\leq 1$ in $[0,\rho]\times[0,T]$.
\item
The equalities $\bar\Psi(r,t)=1$ and
$\frac{\partial\bar\Psi}{\partial r}(r,t)=0$ hold in
$\left[0,\frac{\rho}2\right]\times\left[\tau,T\right]$ and
$\left[0,\frac{\rho}2\right]\times\left[0,T\right]$, respectively.
\item
The estimate $\left|\frac{\partial\bar\Psi}{\partial
t}\right|\leq\frac{\bar C\bar\Psi^{\frac12}}{\tau}$ is satisfied on
$[0,\infty)\times[0,T]$ for some $\bar C>0$, and $\bar\Psi(r,0)=0$
for all $r\in[0,\infty)$.
\item
The inequalities
$-\frac{C_a\bar\Psi^a}\rho\leq\frac{\partial\bar\Psi}{\partial
r}\leq 0$ and $\left|\frac{\partial^2\bar\Psi}{\partial
r^2}\right|\leq\frac{C_a\bar\Psi^a}{\rho^2}$ hold on
$[0,\infty)\times[0,T]$ for every $a\in(0,1)$ with some constant
$C_a$ dependent on $a$.
\end{enumerate}
\end{lemma}

\subsection{Proof of Theorem 1.1}\label{global}

\begin{proof}
Let $f=\log u$. Denote $F=t\left(|\nabla f|^2-f_t\right)$. Fix
$\tau\in(0,T]$ and choose a point $(x_0,t_0)\in M\times[0,\tau]$
where $F$ attains its maximum on $M\times[0,\tau]$. We will show that
\begin{align*}
F(x_0,t_0)\le t_0kn+\frac{C_n}{2}\,.
\end{align*}

If $t_0=0$, then $F(x,t_0)$ is equal to~0 for every $x\in M$ and the
estimate is trivially true. Assume then that $t_0>0$ without loss of generality. We apply Lemma~\ref{lemma_LY} with $k_1=0$, $k_2=k$, $\beta=1$ and $a+2b=1$ and we obtain: 
\begin{align*}
\left(\Delta-\frac\partial{\partial t}\right)F\ge-2\nabla f\nabla
F+\frac{2a}n\frac{F^2}{t_0}-\frac{F}{t_0}-\frac{t_0n}{(1-a)}\,k^2-\frac{\alpha^2(0)n}{1-a}\frac{C^2}{t_0}
\end{align*}
for all $a\in(0,1)$ at the point $(x_0,t_0)$. 

$F$ attains its maximum at $(x_0,t_0)$, which implies that $\triangle F(x_0,t_0)\le0$, $\frac\partial{\partial t}F(x_0,t_0)\ge0$, and $\nabla F(x_0,t_0)=0$. Hence the estimate
\begin{gather*}
\frac{2a}n\frac{F^2}{t_0}-\frac{F}{t_0}-\frac{t_0n}{(1-a)}\,k^2 -\frac{\alpha^2(0)n}{1-a}\frac{C^2}{t_0}\le0 \\
\Updownarrow \\
\frac{2a}nF^2-F-\frac{t_0^2n}{(1-a)}\,k^2 -\frac{\alpha^2(0)n}{1-a}C^2\le0
\end{gather*}
holds at $(x_0,t_0)$.

The quadratic formula implies that 
\begin{align*}
F(x_0,t_0)\le\frac
n{4a}\left(1+\sqrt{1+\frac{8a}{1-a}\cdot (t_0^2k^2+\alpha^2(0)C^2)}\,\right).
\end{align*}
Choosing the following value for $a\in(0,1)$ (which happens to also minimize the RHS expression): 
\[a=\frac{1+\sqrt{2(t_0^2k^2+\alpha^2(0)C^2)}}{1+2\sqrt{2(t_0^2k^2+\alpha^2(0)C^2)}}\]

leads to the following inequality: 
\begin{align*}
F(x_0,t_0)&\leq \frac{n}{4}\cdot \left(1+2\sqrt{2(t_0^2k^2+\alpha^2(0)C^2)}\right) \\
          &\leq \frac{n}{4}\cdot \left(1+4(t_0k+\alpha(0)C)\right)=nt_0k+n\alpha(0)C+\frac{n}{4}
\end{align*}

The fact that $(x_0,t_0)$ is a maximum point for $F$ on $M\times[0,\tau]$
enables us to conclude that
\begin{align*}
F(x,\tau)\le F(x_0,t_0)\le nt_0k+n\alpha(0)C+\frac{n}{4}\leq \tau k n+ n\alpha(0)C+\frac{n}{4}
\end{align*}
for all $x\in M$. Therefore, the estimate
\begin{align*}
\frac{|\nabla u|^2}{u^2} -\frac{u_t}{u}\le kn+\frac{C_n}{2\tau}
\end{align*}
holds at $(x,\tau)$ ($C_n=n/2+4nC\alpha(0)$). Because the number $\tau\in(0,T]$ was chosen
arbitrarily, this holds for all $t\in(0,T]$.
\end{proof}

\subsection{Proof of Theorem 1.2}\label{local}

\begin{proof}
 
Let $f=\log u$ and $F=t\left(|\nabla f|^2-\beta f_t\right)$. By Lemma \ref{lemma_LY}, one has that 
\begin{align*} 
 \triangle F-F_t & \geq -2\nabla f\cdot \nabla F-(|\nabla f|^2-\beta f_t) +\frac{2a\beta t}{n}\left(|\nabla f|^2-f_t\right)^2 -2k_1\beta t|\nabla
f|^2\\
 &-\frac{\beta tn}{2b}\max\left\{k_1^2,k_2^2\right\}-\frac{\beta\alpha^2(0)n}{2b}\frac{C^2}{t} - 2(\beta-1)\alpha(0) C|\nabla f|^2
\end{align*}
 
To make the computation easier to follow, let's assume $k_2=\max\{k_1,k_2\}$ and denote $\alpha(0)=\alpha$. Therefore the above inequality becomes:

\begin{align}\label{lemma_LYM} 
 \triangle F-F_t & \geq -2\nabla f\cdot \nabla F-(|\nabla f|^2-\beta f_t) +\frac{2a\beta t}{n}\left(|\nabla f|^2-f_t\right)^2 -2k_1\beta t|\nabla
f|^2\notag \\
 &-\frac{\beta tnk_2}{2b}-\frac{\beta\alpha^2n}{2b}\frac{C^2}{t} - 2(\beta-1)\alpha C|\nabla f|^2
\end{align}

Following the classical technique, let $\tau\in (0,T]$, $x_0\in M$ and fix $\bar\Psi(r,t)$ satisfying the conditions of Lemma~\ref{lem_cutoff}. Further define $\Psi:M\times[0,T]\to
\mathbb R$ by setting
\begin{align*}
\Psi(x,t)=\bar\Psi(\dist(x,x_0,t),t).
\end{align*}

We will the apply maximum principle to the expression $\left(\frac{\partial}{\partial
t}-\Delta\right)(\Psi F)$ and we will establish~\eqref{sp-tm-loc} at $(x,\tau)$ for $x\in M$ such
that $\dist(x,x_0,\tau)<\frac\rho2$. 

From the inequality \eqref{lemma_LYM} and some manipulations, one obtains: 
\begin{align}\label{aux1_sp-time}
\left(\Delta-\frac{\partial}{\partial t}\right)(\Psi F)& \geq
-2\nabla f\nabla(\Psi F)+2F\nabla f\nabla\Psi \notag
\\ &\hphantom{=}~+\left(\frac{2a\beta t}{n}\left(|\nabla f|^2-f_t\right)^2-(|\nabla f|^2-\beta f_t)  -2k_1\beta t|\nabla f|^2\right)\Psi \notag\\
 &\hphantom{=}~+\left(-\frac{\beta tnk_2}{2b}-\frac{C^2\beta\alpha^2n}{2bt} - 2(\beta-1)\alpha C|\nabla f|^2\right)\Psi \notag \\
&\hphantom{=}~+(\Delta \Psi)F+2\frac{\nabla\Psi}{\Psi}\,\nabla(\Psi
F)-2\frac{|\nabla\Psi|^2}{\Psi}\,F - \frac{\partial\Psi}{\partial
t}\,F
\end{align}
This holds in the region of $B_{\rho,T}$ where $\Psi(x,t)$ is smooth and strictly
positive. Let $(x_1,t_1)$ be a maximum point for the function $\Psi
F$ in the set $\left\{(x,t)\in
M\times[0,\tau]\,|\dist(x,x_0,t)\le\rho\right\}$. We may assume
$(\Psi F)(x_1,t_1)>0$ without loss of generality. Note that if this is
not the case, then $F(x,\tau)\le0$ and~\eqref{sp-tm-loc} is obvious
at $(x,\tau)$ whenever $\dist(x,x_0,\tau)<\frac\rho2$.  We may also
assume that $\Psi(x,t)$ is smooth at $(x_1,t_1)$ due to a standard
trick explained, for example, in~\cite[page 21]{RSSTY94}. Since
$(x_1,t_1)$ is a maximum point, the equalities $\Delta(\Psi
F)(x_1,t_1)\leq 0$, $\nabla(\Psi F)(x_1,t_1)= 0$, and $(\Psi
F)_t(x_1,t_1)\ge0$ hold. Therefore \eqref{aux1_sp-time} gives:
\begin{align}\label{aux2_sp-time}
0& \geq 2F\nabla f\nabla\Psi +\left(\frac{2a\beta t_1}{n}\left(|\nabla f|^2-f_t\right)^2-(|\nabla f|^2-\beta f_t) -2k_1\beta t_1|\nabla f|^2\right)\Psi \notag\\
 &+\left(-\frac{\beta t_1nk_2}{2b}-\frac{C^2\beta\alpha^2n}{2bt_1} - 2(\beta-1)\alpha C|\nabla f|^2\right)\Psi \notag \\
&\hphantom{=}~+(\Delta \Psi)F-2\frac{|\nabla\Psi|^2}{\Psi}\,F
-\frac{\partial\Psi}{\partial t}\,F
\end{align}
at $(x_1,t_1)$. We will use~\eqref{aux2_sp-time} to show that a
certain quadratic expression in $\Psi F$ is non-positive. 

Recalling Lemma~\ref{lem_cutoff}, together with the Laplacian
comparison theorem, one can deduce that
\begin{align*}
-\frac{|\nabla\Psi|^2}{\Psi}&\geq-\frac{C_{\frac12}^2}{\rho^2}\,, \\
\Delta\Psi &\geq -\frac{C_{\frac12}}{\rho^2}
-\frac{C_{\frac12}\Psi^{\frac12}}{\rho}\,(n-1)\sqrt{k_1}\,\coth\left(\sqrt{k_1}\,\rho\right)\ge
-\frac{d_1}{\rho^2} -\frac{d_1\Psi^{\frac12}}{\rho}\,\sqrt{k_1}
\end{align*}
at the point $(x_1,t_1)$ with $d_1$ a positive constant depending on
$n$. 

We will now find a bound for $(\Psi_t)(x_1,t_1)$. By definition of $\Psi$:
\begin{align}\label{tm-der-aux1}
(\Psi_t)(x_1,t_1)&=\frac{\partial \bar\Psi}{\partial
t}(\dist(x_1,x_0,t_1),t_1)\notag \\ &\hphantom{=}~+
\frac{\partial \bar\Psi}{\partial
r}(\dist(x_1,x_0,t_1),t_1)\left(\frac{\partial}{\partial
t}\dist(x_1,x_0,t_1)\right).
\end{align}

Using the fact that the function~$\bar\Psi(r,t)$
satisfies the conditions listed in Lemma~\ref{lem_cutoff}, the
inequality
\begin{align}\label{tm-der-aux3}
\left|\frac{\partial \bar\Psi}{\partial
r}(\dist(x_1,x_0,t_1),t_1)\right|\leq
\frac{C_{\frac12}}{\rho}\,\Psi^{\frac12}(x_1,t_1)
\end{align}
holds with $C_{\frac12}>0$. For the estimate of the derivative of
the distance, notice that: 
\begin{align}\label{tm-der-aux2}
\left|\frac{\partial}{\partial
t}\dist(x_1,x_0,t_1)\right|\le\sup\int\limits_{0}^{\dist(x_1,x_0,t_1)}\left|\Ric\left(\frac
d{ds}\zeta(s),\frac d{ds}\zeta(s)\right)\right|ds \leq
k_2\dist(x_1,x_0,t_1)\leq k_2\rho.
\end{align}
Here $\Ric$ represents the Ricci curvature of $g(x,t_1)$ and the supremum is taken over all the minimal geodesics $\zeta(s)$, with respect to $g(x,t_1)$, that connect~$x_0$ to~$x_1$
and are parametrized by arclength; see, e.g.,~\cite[Proof~of~Lemma~8.28]{BCPLLN06}. 

Moreover, again by Lemma~\ref{lem_cutoff} 
\begin{align*}
\left|\frac{\partial\bar\Psi}{\partial t}(\dist(x_1,x_0,t_1),t_1)\right|\leq \frac{c_2\bar\Psi^{1/2}}{\tau}\leq \frac{c_2\Psi^{1/2}}{\tau}
\end{align*}
for a positive constant $c_2$. 

Therefore, we can conclude that there exists $\bar C>0$ such that the inequality
\begin{align*}
-\frac{\partial\Psi}{\partial t}&\geq -\frac{\bar
C\Psi^{\frac12}}{\tau}-C_{\frac12}k_2\Psi^{\frac12}
\end{align*}
holds true. Combining this with~\eqref{aux2_sp-time}, we find the estimate
\begin{align*}
0& \geq -2F|\nabla f||\nabla\Psi| + \left(\frac{2a\beta t_1}{n}\left(|\nabla f|^2-f_t\right)^2-(|\nabla f|^2-\beta f_t) -2k_1\beta t_1|\nabla f|^2\right)\Psi \\
&+\left(-\frac{\beta t_1nk_2}{2b}-\frac{C^2\beta\alpha^2n}{2bt_1} - 2(\beta-1)\alpha C|\nabla f|^2\right)\Psi \\
&\hphantom{=}~+d_2\left(-\frac1{\rho^2}-\frac{\Psi^{\frac12}}{\rho}\,\sqrt{k_1}-\frac{\Psi^{\frac12}}{\tau}
-k_2\Psi^{\frac12}\right)F
\end{align*}
at $(x_1,t_1)$. Here, $d_2$ is equal to
$\max\left\{3d_1,C_{\frac12},3C_{\frac12}^2,\bar C\right\}$. Let us now multiply by $t_1\Psi$, rearrange the terms and obtain: 
\begin{align}\label{aux3_sp-time}
0& \geq -2t_1F\frac{C_{\frac12}\Psi^\frac32}{\rho}\,|\nabla f|
+\frac{2t_1^2}n\left(a\beta\left(\Psi|\nabla f|^2-\Psi f_t\right)^2
-nk_1\beta\Psi^2|\nabla f|^2-\frac{n^2\beta}{4b} k_2^2\Psi^2\right) \notag\\
 & +\left(-\frac{C^2\beta\alpha^2n^2}{4bt_1^2}\Psi^2-\frac{(\beta-1)\alpha C|\nabla f|^2}{t_1}\Psi^2\right) \notag\\
&\hphantom{=}~+d_2t_1\left(-\frac1{\rho^2}-\frac{\sqrt{k_1}}{\rho}-\frac1\tau-
k_2\right)(\Psi F)-\Psi F
\end{align}
at $(x_1,t_1)$. Our next step is to estimate the first two terms in
the right-hand side. The procedure is standard (see \cite{PLSTY86}). 

Define $y=\Psi|\nabla f|^2$ and $z=\Psi f_t$. It follows that
$y^{\frac12}(y-\beta z)=\frac{\Psi^{\frac32}F|\nabla f|}{t}$ when
$t\ne0$, which yields
\begin{align*}
-2t_1F\frac{C_{\frac12}\Psi^\frac32}{\rho}\,|\nabla f|
+\frac{2t_1^2}n\left(a\beta\left(\Psi|\nabla f|^2-\Psi f_t\right)^2
-nk_1\beta\Psi^2|\nabla f|^2-\frac{n^2\beta}{4b} k_2^2\Psi^2\right) \\
+\left(- \frac{C^2\beta\alpha^2n^2}{4bt_1^2}\Psi^2-\frac{(\beta-1)\alpha C|\nabla f|^2}{t_1}\Psi^2\right) \\
= \frac{2t_1^2}{n}\,\left(-\frac{nC_{\frac12}}{\rho}\,y^{\frac12}(y-\beta z)+a\beta(y-z)^2- nk_1\beta y-\frac{n^2\beta}{4b}\,k_2^2\Psi^2- \frac{C^2\beta\alpha^2n^2}{4bt_1^2}\Psi^2-\frac{(\beta-1)\alpha C}{t_1}y\Psi\right).
\end{align*}
Observe that \begin{align*}(y-z)^2=\frac{1}{\beta^2}\,(y-\beta
z)^2+\frac{(\beta-1)^2}{\beta^2}\,y^2+\frac{2(\beta-1)}{\beta^2}\,y(y-\beta
z)\end{align*} and plug this into the above:
\begin{align*}\frac{2t_1^2}{n}\,\left(-\frac{nC_{\frac12}}{\rho}\,y^{\frac12}(y-\beta z)+a\beta(y-z)^2- nk_1\beta y-\frac{n^2\beta}{4b}\,k_2^2\Psi^2- \frac{C^2\beta\alpha^2n^2}{4bt_1^2}\Psi^2-\frac{(\beta-1)\alpha C}{t_1}y\Psi\right)\\
= \frac{2t_1^2}{n}\,\left(-\frac{nC_{\frac12}}{\rho}\,y^{\frac12}(y-\beta z)+\frac{a(y-\beta z)^2}{\beta}+\frac{a}{\beta}(\beta-1)^2y^2+\frac{2(\beta-1)a}{\beta}y(y-\beta z)\right)\\
 +\frac{2t_1^2}{n}\left(-nk_1\beta y-\frac{n^2\beta}{4b}\,k_2^2\Psi^2- \frac{C^2\beta\alpha^2n^2}{4bt_1^2}\Psi^2-\frac{(\beta-1)\alpha C}{t_1}y\Psi\right)
\end{align*}
Using the inequality $\kappa_1v^2-\kappa_2v\ge -\frac{\kappa_2^2}{4\kappa_1}$, which holds for $\kappa_1,\kappa_2>0$ and $v\in\mathbb R$, we obtain:

\begin{align*}
\frac{2(\beta-1)a}{\beta}y(y-\beta z)-\frac{nC_{\frac12}}{\rho}\,y^{\frac12}(y-\beta z)\geq -\frac{n^2C_{\frac{1}{2}}^2\beta(y-\beta z)}{\rho^28a(\beta-1)} \\
\frac{a}{\beta}(\beta-1)^2y^2-\left(nk_1\beta+\frac{(\beta-1)\alpha C}{t_1}\Psi\right)y\geq -\frac{\beta\left(nk_1\beta+\frac{(\beta-1)\alpha C}{t_1}\Psi\right)^2}{4a(\beta-1)^2}
\end{align*}

It then follows that: 

\begin{align*}
-2t_1F\frac{C_{\frac12}\Psi^\frac32}{\rho}\,|\nabla f|
+\frac{2t_1^2}n\left(a\beta\left(\Psi|\nabla f|^2-\Psi f_t\right)^2
-nk_1\beta\Psi^2|\nabla f|^2-\frac{n^2\beta}{4b} k_2^2\Psi^2\right)\\
+\left(- \frac{C^2\beta\alpha^2n^2}{4bt_1^2}\Psi^2-\frac{(\beta-1)\alpha C|\nabla f|^2}{t_1}\Psi^2\right)\\
\geq \frac{2t_1^2}{n} \left(\frac{a(y-\beta z)^2}{\beta}-\frac{n^2\beta}{4b}\,k_2^2\Psi^2- \frac{C^2\beta\alpha^2n^2}{4bt_1^2}\Psi^2-\frac{n^2C_{\frac{1}{2}}^2\beta(y-\beta z)}{\rho^28a(\beta-1)} - \frac{\beta\left(nk_1\beta+\frac{(\beta-1)\alpha C}{t_1}\Psi\right)^2}{4a(\beta-1)^2}\right)
\end{align*}

Recall that $t(y-\beta z)=\Psi F$ so inequality \eqref{aux3_sp-time} becomes: 
\begin{align*}
0& \geq \frac{2a}{n\beta}\,(\Psi
F)^2+\left[-\frac{nd_3t_1}{\rho^2}\left(\frac{\beta}{a(\beta-1)}
+1+\rho\sqrt{k_1}+\frac{\rho^2}{\tau}+\rho^2k_2\right)-1\right](\Psi
F)\\
&\hphantom{=}~-\frac{\beta\left(nk_1\beta+\frac{(\beta-1)\alpha C}{t_1}\Psi\right)^2}{2an(\beta-1)^2}\,t_1^2-\frac{\beta n}{2b}\,t_1^2 k_2^2\Psi^2 - \frac{C^2\beta\alpha^2n}{2b}\Psi^2\\ 
&\geq \frac{2a}{n\beta}\,(\Psi
F)^2+\left(-\frac{d_4t_1}{\rho^2}\left(\frac{\beta}{a(\beta-1)}
+\frac{\rho^2}{\tau}+\rho^2\bar k\right)-1\right)(\Psi
F)\\
&\hphantom{=}~-\frac{\beta\left(nk_1\beta+\frac{(\beta-1)\alpha C}{t_1}\Psi\right)^2}{2an(\beta-1)^2}\,t_1^2-\left(\frac{\beta n}{2b}\,t_1^2 k_2^2+\frac{C^2\beta\alpha^2n}{2b}\right)\Psi^2\end{align*} 
at $(x_1,t_1)$ with
$d_3=\max\{4d_2,C_{1/2}\}$, $d_4=nd_3$ and $\bar k=\max\{k_1,k_2\}$. The expression in the last two lines is a polynomial in
$\Psi F$ of degree~2. The quadratic formula yields:
\begin{align*}
\Psi F &\leq \frac{n\beta}{4a}\left[2\frac{d_4t_1}{\rho^2}\left(\frac{\beta}{a(\beta-1)}
+\frac{\rho^2}{\tau}+\rho^2\bar k\right)+2\right] \\
       & +\frac{n\beta}{4a}\left[\frac{2}{n(\beta-1)}\left(nk_1b+\frac{(\beta-1)\alpha C}{t_1}\Psi\right)\,t_1+2\sqrt{\frac{a}{b}\,}\,(t_1 k_2+C\alpha)\Psi\right] \\
       & = \frac{n\beta}{2a}\left[\frac{d_4t_1}{\rho^2}\left(\frac{\beta}{a(\beta-1)}
+\frac{\rho^2}{\tau}+\rho^2\bar k\right)+1+
\frac{k_1bt_1}{(\beta-1)}+\left(\frac{\alpha C}{n}+\sqrt{\frac{a}{b}\,}\,(t_1 k_2+C\alpha)\right)\Psi\right] 
\end{align*}
at $(x_1,t_1)$. We will now use this estimate to obtain a bound on
$F(x,\tau)$ for an appropriate range of $x\in M$.

By definition $\Psi(x,\tau)=1$ whenever $\dist(x,x_0,\tau)<\frac\rho2$. Moreover, $(x_1,t_1)$ is a maximum point for $\Psi F$ in the set $\left\{(x,t)\in M\times[0,\tau]\,|\dist(x,x_0,t)\le\rho\right\}$. Hence 
\begin{align*}
F(x,\tau)&=(\Psi F)(x,\tau)\le (\Psi F)(x_1,t_1) \\ &\le \frac{n\beta}{2a}\cdot\frac{d_4\tau}{\rho^2}\left(\frac{\beta}{a(\beta-1)}
+\frac{\rho^2}{\tau}+\rho^2\bar k\right)+ \frac{n\beta}{2a}\left(1+\frac{\alpha C}{n}+\sqrt{\frac{a}{b}}\alpha C\right)+
 \frac{n\beta}{2a}\cdot\frac{k_1b\tau}{(\beta-1)}+ \frac{n\beta k_2}{2\sqrt{ab}}\tau
\end{align*}
for all $x\in M$ such that $\dist(x,x_0,\tau)<\frac\rho2$. Since
$\tau\in(0,T]$ was chosen arbitrarily, this formula implies 
\begin{align*}
\left(|\nabla f|^2-\beta f_t\right)(x,t) & \le \frac{n\beta}{2a}\cdot\frac{d_4}{\rho^2}\left(\frac{\beta}{a(\beta-1)}+\frac{\rho^2}{t}+\rho^2\bar k\right) \\ 
    & + \frac{n\beta}{2a t}\left(1+\frac{\alpha C}{n}+\sqrt{\frac{a}{b}}\alpha C\right)+
 \frac{n\beta}{2a}\cdot\frac{k_1b}{(\beta-1)}+\frac{n\beta k_2}{2\sqrt{ab}} \\
  & = \frac{n\beta}{2a}\cdot\frac{d_4}{\rho^2}\cdot\frac{\beta}{a(\beta-1)}
 + \frac{n\beta}{2at}\left(d_4+1+\frac{\alpha C}{n}+\sqrt{\frac{a}{b}}\alpha C\right) \\
 & + \frac{n\beta}{2a}\cdot\frac{k_1b}{(\beta-1)}+ \frac{n\beta}{2}\left(\frac{d_4\bar k}{a}+\frac{k_2}{\sqrt{ab}}\right)
\end{align*}

for $(x,t)\in B_{\frac\rho2,T},$, as long as $t\ne0$. 

If we set $a=\frac1{2\beta}$ and $b=\frac1{4\beta}$, then the above becomes: 

\begin{align*}
\left(|\nabla f|^2-\beta f_t\right)(x,t) & \le n\beta^2\cdot\frac{d_4}{\rho^2}\cdot\frac{2\beta^2}{(\beta-1)} + \frac{n\beta^2}{t}\left(d_4+1+\frac{\alpha C}{n}+\sqrt{2}\alpha C\right) \\
       & + \frac{n\beta k_1}{4(\beta-1)}+ n\beta^2\left(d_4\bar k+2k_2\right)
\end{align*}

Define the constant $C'$ as $\max\{2nd_4, n\left(d_4+1+\frac{\alpha C}{n}+\sqrt{2}\alpha C\right), n(d_4+2)\}$, estimate~\eqref{sp-tm-loc} will follow:
\begin{align*}
\left(|\nabla f|^2-\beta f_t\right)(x,t) & \le C'\beta^2\cdot\left(\frac{\beta^2}{\rho(\beta-1)}+\frac{1}{t}+\max\{k_1,k_2\}\right)+\frac{n\beta k_1}{4(\beta-1)}
\end{align*}
for $(x,t)\in B_{\frac\rho2,T},$, as long as $t\ne0$.
\end{proof}

\section{Harnack inequalities}\label{trei}

In this section, we will prove two Harnack inequalities as a direct application to the gradient estimates obtained in section \ref{doi}, as for example,~\cite[Chapter~IV]{RSSTY94}. One can find similar Harnack inequalities for the heat equation under the Ricci flow in the
papers~\cite{CG02,LN04,MBXCAP09}. 

Given $x_1,x_2\in M$ and $t_1,t_2\in(0,T)$ satisfying $t_1<t_2$,  define
\begin{align*}
\Gamma(x_1,t_1,x_2,t_2)=\inf\int\limits_{t_1}^{t_2}\left|\frac{d}{dt}\gamma(t)\right|^2dt.
\end{align*}
where the infimum is taken over the set $\Theta(x_1,t_1,x_2,t_2)$ of all
the smooth paths $\gamma:[t_1,t_2]\to M$ that connect~$x_1$
to~$x_2$ and the norm~$|\cdot|$ is calculated at time $t$. 

The way we will prove the Harnack inequalities is to use the above gradient estimates together with the next lemma: 

\begin{lemma}\label{lemma_Harnack}
Suppose $\big(\phi(x,t),g(x,t)\big)_{t\in[0,T]}$ is a solution to the $(RH)_\alpha$ flow
\eqref{RH_flow_intro} such that $M$ is a complete manifold. Let $u: M\times [0,T]\to
\mathbb{R}$ be a smooth positive function satisfying the heat
equation~\eqref{heat_intro}. Define $f=\log u$ and assume that
\begin{align*}\frac{\partial f}{\partial t}\geq\frac{1}{A_1}\left(|\nabla f|^2-A_2-\frac{A_3}{t}\right),
\qquad x\in M,~t\in(0,T],\end{align*} for some $A_1,A_2,A_3>0$. Then
the inequality
\begin{align*}
u(x_2,t_2)\geq
u(x_1,t_1)\left(\frac{t_2}{t_1}\right)^{-\frac{A_3}{A_1}}
\exp\left(-\frac{A_1}4\Gamma(x_1,t_1,x_2,t_2)-\frac{A_2}{A_1}\,(t_2-t_1)\right)
\end{align*}
holds for all $(x_1,t_1),(x_2,t_2)\in M\times(0,T)$ such that $t_1<t_2$.
\end{lemma}

\begin{proof}
The technique used to prove this result is traditional; see, for
example,~\cite[Chapter~IV]{RSSTY94} and~\cite{MBXCAP09, XCRH09}. Let
$\gamma(t)\in\Theta(x_1,t_1,x_2,t_2)$ be a path connecting $(x_1,t_1)$ and $(x_2,t_2)$. Taking the time derivative at the point $(\gamma(t),t)$, one obtains:
\begin{align*}
\frac{d}{dt}f(\gamma(t),t)&=\nabla
f(\gamma(t),t)\frac{d}{dt}\gamma(t)+\frac{\partial}{\partial s}f(\gamma(t),s)|_{s=t} \\
&\geq-|\nabla f(\gamma(t),t)|\left|\frac{d}{dt}\gamma(t)\right|
+\frac{1}{A_1}\left(|\nabla f(\gamma(t),t)|^2-A_2-\frac{A_3}{t}\right) \\
&\geq
-\frac{A_1}4\left|\frac{d}{dt}\gamma(t)\right|^2-\frac1{A_1}\left(A_2+\frac{A_3}{t}\right),\qquad
t\in[t_1,t_2].
\end{align*}
where for the last step one applies the inequality
$\kappa_1v^2-\kappa_2v\geq-\frac{\kappa_2^2}{4\kappa_1}$ (for
$\kappa_1,\kappa_2>0$ and $v\in\mathbb{R}$). 

Next, let us integrate from $t_1$ to $t_2$
\begin{align*}
f(x_2,t_2)-f(x_1,t_1)&=
\int\limits_{t_1}^{t_2}\frac{d}{dt}f(\gamma(t),t)\,dt
\\ &\ge-\frac{A_1}4\int\limits_{t_1}^{t_2}\left|\frac{d}{dt}\gamma(t)\right|^2\,dt
-\frac{A_2}{A_1}(t_2-t_1)-\frac{A_3}{A_1}\ln\frac{t_2}{t_1}.
\end{align*}
Now the conclusion follows by exponentiating.
\end{proof}

We will obtain two Harnack inequalities for~\eqref{RH_flow_intro}--\eqref{heat_intro}, one for noncompact manifolds, one for compact. 

\begin{theorem}
Let $\big(g(x,t),\phi(x,t)\big)_{t\in[0,T]}$ be a solution to the
$(RH)_\alpha$ flow~\eqref{RH_flow_intro}, where $M^n$ is complete and $\alpha(t)$ is a non-increasing function in time, bounded from below by $\bar\alpha$. Suppose $-k_1g(x,t)\le\Ric(x,t)\leq k_2g(x,t)$ for some $k_1,k_2>0$ and that $0\leq\nabla_i\phi\nabla_j\phi\leq \frac{C}{t}g_{ij}$, for all $i,j\in\{1,2,...,n\}$ ($C$ being a constant depending on $n$ and $\bar\alpha$). Consider a smooth positive function $u:M\times[0,T]\to\mathbb R$ solving the heat equation $\triangle u=u_t$. Given $\beta>1$, the estimate
\begin{align*}
u(x_2,t_2)\geq u(x_1,t_1)\left(\frac{t_2}{t_1}\right)^{-C'\beta}
\exp\left(-\frac{\beta}4\Gamma(x_1,t_1,x_2,t_2)-\left(C'\beta
\max\left\{k_1,k_2\right\}+\frac{n\beta k_1}{4(\beta-1)}\right)(t_2-t_1)\right)
\end{align*}
holds for all $(x_1,t_1),(x_2,t_2)\in M\times(0,T)$ such that $t_1<t_2$. The constant $C'$ comes from Theorem~\ref{thm_sp-tm-local}.
\end{theorem}

\begin{proof}
Let $\rho$ go to infinity in~\eqref{sp-tm-loc} and obtain that
\begin{align*}
\frac{u_t}{u}\ge\frac1\beta\left(\frac{|\nabla
u|^2}{u^2}-\frac{C'\beta^2}t-\left(C'\beta^2\max\left\{k_1,k_2\right\}+\frac{n\beta^2k_1}{4(\beta-1)}\right)\right)
\end{align*}
on $M\times(0,T]$. To get the conclusion, apply Lemma~\ref{lemma_Harnack}.
\end{proof}

\begin{theorem}
Let $\big(g(x,t),\phi(x,t)\big)_{t\in[0,T]}$ is a solution to the $(RH)_\alpha$
flow \eqref{RH_flow_intro}, where the manifold $M^n$ is compact and $\alpha(t)$ is a non-increasing function in time, bounded from below by $\bar\alpha$. Assume that $0\le\Ric(x,t)\leq kg(x,t)$ for some $k>0$ and all $(x,t)\in M\times[0,T]$. Moreover suppose that $0\leq\nabla_i\phi\nabla_j\phi\leq \frac{C}{t}g_{ij}$, for all $i,j\in\{1,2,...,n\}$ ($C$ being a constant depending on $n$ and $\bar\alpha$).  Let $u$ be a smooth
positive function $u:M\times[0,T]\to\mathbb R$ satisfying the heat
equation $u_t=\triangle u$. 
The estimate
\begin{align*}
u(x_2,t_2)\geq u(x_1,t_1)\left(\frac{t_2}{t_1}\right)^{-\frac{C_n}{2}}
\exp\left(-\frac14\Gamma(x_1,t_1,x_2,t_2)-kn(t_2-t_1)\right)
\end{align*}
holds for all $(x_1,t_1)\in M\times(0,T)$ and $(x_2,t_2)\in
M\times(0,T)$ as long as $t_1<t_2$. The constant $C_n$ is the same as in Theorem\ref{thm_sp-tm-global}. 
\end{theorem}

\begin{proof}
Theorem~\ref{thm_sp-tm-global} implies \begin{align*}
\frac{u_t}u\geq \frac{|\nabla u|^2}{u^2}-kn-\frac{C_n}{2t},\qquad x\in
M,~t\in(0,T].
\end{align*}
Use again Lemma~\ref{lemma_Harnack} to finish the proof.
\end{proof}

\subsection{Concluding remarks}

We have obtained Li-Yau gradient estimates (and as an application Harnack inequalities) for the solution of the heat equation under the Ricci-harmonic flow, under a slightly stronger assumption on the covariant derivatives of the harmonic map. The results show a similar behavior of the heat function as in the classical (non-evolving) case. In the future, one may want to study the behavior in an even more general setting, for example only a lower bound on the Ricci curvature and no condition on the covariant derivatives of the harmonic map.

\bibliographystyle{unsrt}
\bibliography{Gradient-estimates}

\end{document}